\documentclass{amsart}

\usepackage[margin=1in]{geometry}
\usepackage{amsmath,amsthm,amssymb,array,enumerate,parskip,graphicx, etoolbox,tabularx}
\usepackage[all]{xy}

\usepackage{xcolor} 
\usepackage{tikz-cd}

\usepackage[pagebackref]{hyperref}
\definecolor{dark-red}{rgb}{0.5,0.15,0.15}
\definecolor{dark-blue}{rgb}{0.15,0.15,0.6}
\definecolor{dark-green}{rgb}{0.15,0.6,0.15}
\hypersetup{
    colorlinks, linkcolor=dark-red,
    citecolor=dark-blue, urlcolor=dark-red
}

\usepackage[capitalise]{cleveref}
\usepackage{comment}

\newcommand{\Z}{\mathbb{Z}} 
\newcommand{\Q}{\mathbb{Q}} 
\newcommand{\C}{\mathbb{C}} 

\newcommand \ZZ {{\mathbb Z}}

\newcommand \QQ {{\mathbb Q}}

  \newcommand{\ilim}{\mathop{\varprojlim}\limits}

\newtheorem{theorem}{Theorem}[section]

 \theoremstyle{definition}
\newtheorem{remark}[theorem]{Remark} 
\newtheorem{definition}[theorem]{Definition}

\newtheorem{fact}[theorem]{Fact}

\title[$Q_8$ Origami]{Arithmetic of quaternion origami}
\author{Rachel Davis}
\address{University of Wisconsin-Madison}
\email{rachel.davis@wisc.edu}
\author{Edray Herber Goins}
\address{Purdue University}
\email{goins@purdue.edu}

 \thanks{We would like to thank Purdue University where this project began. Thank you to Nigel Boston for the original questions about outer Galois representations for non-abelian covers and to Donu Arapura for ideas about these covers. Also, thank you to Jayadev Athreya, Rachel Pries, Jeremy Rouse, Tony Shaska, and Vesna Stojanoska for helpful discussions. }

\begin{document}
\begin{abstract}
We study origami $f: C \rightarrow E$ with $G$-Galois cover $Q_8$. For a point $P \in E(\Q) \backslash \left\{ \mathcal{O} \right\}$, 
we study the field obtained by adjoining to $\Q$ the coordinates of all of the preimages of $P$ under $f$. We find a defining polynomial, 
$f_{E, Q_8,P}$, for this field and study its Galois group. We give an isomorphism depending on $P$ between a certain subfield of this 
field and a certain subfield of the 4-division field of the elliptic curve.\\
MSC2010: 11R32, 12E05, 12F05, 13P15, 14H30, 14H45, 14H52, 20B35 \\
Keywords: Cover of curve, division polynomial, elliptic curve, generic Galois extension, origami, outer Galois representation, quaternion group, resultant \end{abstract}

\maketitle

\section{Introduction}
Let $E$ be an elliptic curve over $\Q$ with affine Weierstrass model given by $y^2=x^3+ax+b$ and with point at infinity given in projective coordinates by $\mathcal{O}=(0:1:0)$. An origami is used as a mathematical term for covers of an elliptic curve, ramified above at most one point. Quaternion origami, specifically, have been studied by Herrlich and Schmith{\"u}sen \cite{HS}. 

Origami are the $G$-structures of elliptic curves in the notation of Chen \cite{Chen}. We will consider $G=Q_8$-structures on $X=E - \left\{ \mathcal{O} \right\}$, in the notation of \cite{Chen} and \cite{ChenDeligne}. By \cite{win2p}, the image of Grothendieck's outer Galois representation associated to $E$, in this case, is isomorphic to the image of the 2-adic Tate module representation associated to $E$. Chen and Deligne \cite[Theorem 5.1.1]{ChenDeligne}, prove that any metabelian 2-generated group, such as $Q_8$, corresponds to congruence. Their result is also related to the images of Grothendieck's outer Galois representations associated to $E$. We study the quaternion origami curve in order to give an analogue of a division polynomial for elliptic curves, $f_{E, Q_8, P}$, to study the Galois theory of the splitting field of this polynomial, and to understand field-theoretic implications of outer Galois representation results \cite{win2p}, \cite{ChenDeligne} in the smallest non-abelian, 2-group case. The case $G=S_3$ is also metabelian and can be understood as an example of the dihedral covers case.

We define and work with fields extensions $\Q([4]^{-1}P)/\Q$ and the splitting field of $f_{E, Q_8,P}$ over $\Q$ and prove that there is necessarily more overlap than given by initial analysis. We expect some overlap from the outset because both fields extensions will always contain the field extension $\Q([2]^{-1}P)/\Q$. Let $g=x^4-4\Delta x-12 a \Delta$ where $\Delta=-16(4a^3+27b^2)$. The splitting field of $g$ over $\Q$ is a subfield of $\Q(E[4])/\Q$ which is a subfield of $\Q([4]^{-1}P)/\Q$. The main theorem of this article is that the splitting field of $g$ over $\Q$ is also contained in the splitting field of $f_{E, Q_8,P}$ over $\Q$. Once there are two common $S_4$-extensions of $\Q$ containing the 2-division field of $E$, there is also a third common $S_4$-extension of $\Q$ \cite[Proposition 1.1]{BF}, \cite[Theorem 2.4]{CD}.

There is a Galois representation associated to $f_{E, Q_8, P}$ and a quotient of this representation obtained by forgetting the extra information of the point, $P$. In the $Q_8$ case, one consequence of the main result is that information from the Galois representation that is independent of the point is already encoded in the mod 4 representation of the elliptic curve. We already knew, by \cite{win2p} that the information was completely encoded by the 2-adic representation. The result in this paper also gives an explicit isomorphism (depending on $P$) between 2 quartic fields: one field depends on $E$ and $P$ and the other field (obtained by adjoining the roots of $g$ to $\Q$), depends only on $E$. 

For $G=Q_8$, the Galois group of the splitting field of $f_{E, Q_8, P}$ over $\Q$ is related to $\mathrm{Hol}(Q_8)$, the holomorph of $Q_8$ see \ref{holomorph}. In general, for $G$ a fixed finite group, Saltman has related lifting and approximation problems for $G$-Galois extensions to Noether's problem and to generic Galois extensions \cite{Saltmanp3}. Noether's problem implies more than the appearance of $G$ as a Galois group. It further gives ways in which $G$-Galois extensions can be `parameterized' See \cite{Noether} and Saltman's survey article \cite{Saltmansurvey}.

\section{Tate module representations}
For $E$ over $\Q$ with Weierstrass model $y^2=x^3+ax+b$, the $n$-division polynomials are a recursively defined sequence in $\Z[x,y,a,b]$ $\psi_0=0$, $\psi_1=1$, $\psi_2=2y$, $\psi_3=3x^4+6ax^2+12bx-a^2$, $\psi_4=4y(x^6+ax^4+20bx^3-5a^2x^2-4abx-8b^2-a^3), \ldots, \psi_{2m+1}=\psi_{m+2}\psi_m^3-\psi_{m-1}\psi_{m+1}^3$ for $m \geq 2$, $\psi_{2m}=\left( \frac{\psi_m}{2y} \right)(\psi_{m+2}\psi_{m-1}^2-\psi_{m-2}\psi_{m+1}^2)$ for $m \geq 3$. See \cite{Silverman}, \cite{Washington} for background on elliptic curves and division polynomials.

Let $P=(x_P,y_P) \in E(\Q) \setminus \left\{ \mathcal{O} \right\}$. Define $$[n]P=P \oplus P \oplus \ldots \oplus P = \left( \frac{\phi_n(x_P)}{\psi_n^2(x_P)}, \frac{\omega_n(x_P,y_P)}{\psi_n^3(x_P, y_P)} \right)$$ where $\phi_n=x\psi_n^2-\psi_{n+1}\psi_{n-1}$, $\omega_n=\frac{\psi_{n+2}\psi_{n-1}^2-\psi_{n-2}\psi_{n+1}^2}{4y}$.

Define the $n$-division points of $E$ as $E[n]=\left\{ P \in E(\overline{Q}) | [n]P=\mathcal{O} \right\}$. Denote the $n$-division field by $\Q(E[n])$. This is the field obtained by adjoining all of the $n$-division points of $E$ to $\Q$. Fix a prime $\ell$ and define the Tate module to be the inverse limit $T_{\ell}(E) = \ilim_n E[\ell^n]$.

Fix an algebraic closure $\overline{\Q}$ of $\Q$. Let $G_{\Q}$ denote the absolute Galois group of $\overline{\Q}$ over $\Q$. Then, $G_{\Q}$ permutes the $n$-division points. Let $\overline{\rho}_{E,n}:G_{\Q} \rightarrow \mathrm{Aut}(E[\ell^n])$ denote the representation of the absolute Galois group on the $n$-division points. 

The $[\ell^n]$-division points are isomorphic to $(\Z/ \ell^n \Z)^2$, so after a choice of basis, $\mathrm{Aut}(E[\ell^n])$ can be identified with $\mathrm{GL}_2( \Z/ \ell^n\Z)$. 

Taking inverse limits, yields the $\ell$-adic Tate representation, which we will denote $\rho_{E,\ell}: G_{\Q} \rightarrow \mathrm{GL}_2(\Z_{\ell})$. 

Serre shows the following lifting and surjectivity results for the $\ell$-adic representations (see \cite{Serre68}, \cite{Serre72}):

\begin{itemize}
\item $\ell=2$ The map $\rho_{E,2}$ is surjective if and only if $\overline{\rho}_{E,8}$ is surjective.
\item $\ell =3$ The map $\rho_{E,3}$ is surjective if and only if $\overline{\rho}_{E,9}$ is surjective.
\item $\ell \geq 5$ The map $\rho_{E,\ell}$ is surjective if and only if $\overline{\rho}_{E,\ell}$ is surjective.
\end{itemize}

\section{Preimages of a point}

Now, fix a point $P=(z,w) \in E(\Q)$. Consider the set 
$$ V=[n]^{-1}P = \left\{ Q \in E(\overline{\Q}) | [n]Q =P \right\}$$

For example, when $P=\mathcal{O}$, this is the set of $n$-division points. This is no longer group in general, but we can still adjoin the coordinates of such points to $\Q$ and find the Galois group of the extension. From now on, let $P \in E(\Q) \backslash \left\{ \mathcal{O} \right\}$.

The Galois group of $\Q([n]^{-1}P)$ over $\Q$ is a subgroup of the affine general linear group

$$ 1 \rightarrow (\Z/n\Z)^2 \rightarrow \mathrm{AGL}_2(\Z/n\Z) \rightarrow \mathrm{GL}_2(\Z/n\Z) \rightarrow 1 $$

e.g. for $n=2$

$$ 1 \rightarrow (\Z/2\Z)^2 \rightarrow S_4 \rightarrow S_3 \rightarrow 1$$

$$ \mathrm{AGL}_2(\Z/n\Z) = \left\{  \left( \begin{array}{ccc}
a & b & e \\
c & d & f \\
0 & 0 & 1 \end{array} \right) : a,b,c,d,e,f \in (\Z / n \Z) \text{ and } ad-bc \ne 0  \right\}$$

Fxing $\ell$ and taking an inverse limit yields a representation

$$ \rho_{E,\ell, P}: G_{\QQ} \rightarrow \mathrm{Aut}(\ZZ_{\ell}^2) \simeq \mathrm{AGL}_2(\ZZ_{\ell})$$

Affine general linear representations associated to elliptic curves have been studied, for example, see \cite{Cremona1}, \cite{Gorman}, and \cite{somos}. 

\subsection{Division by 2}

Fixing $n=2$, we find that $\phi_2(x)=4x(x^3+ax+b)-(3x^4+6ax^2+12bx-a^2)=
x^4-2ax^2-8bx+a^2$. Fix $P=(z,w)$. If $Q=(x,y) \in [2]^{-1}P$, then 
$$ (z,w) = \left( \frac{x^4-2ax^2-8bx+a^2}{4(x^3+ax+b)}, \frac{(x^6+5ax^4+20bx^3-5a^2x^2-4abx-8b^2-a^3)}{2(x^3+ax+b)} \right)$$

Solving, we find that $f_x=f_{E,[2]^{-1}P,x}(x)=(x^4-2ax^2-8bx+a^2)-z(4(x^3+ax+b))=0$ and $f_{x,y}=f_{E,[2]^{-1},P,x,y}(x,y)=(x^6+5ax^4+20bx^3-5a^2x^2-4abx-8b^2-a^3)-w(8y(x^3+ax+b))$. Bayer and Frey \cite[p.402]{BF} studied $f_x$ in order to study Galois representations of octahedral type and 2-covering of elliptic curves.

Example: Let $E$ be the elliptic curve with Cremona reference `83a1'. Then $a=1269$, $b=-10746$, and $\Delta=2^{12}3^{12}83$. By \cite{Dokchitsers}, $\overline{\rho}_{E,2}$ is surjective if and only if $x^3+ax+b$ is surjective and $\Delta \not \in \mathbb{Q}^2$ The Galois group of the splitting field of $x^3+ax+b$ over $\Q$ is $S_3 \simeq \mathrm{GL}_2(\Z/2\Z)$. Therefore, $\overline{\rho}_{E,2}$ is surjective. In fact, $\overline{\rho}_{E,8}$ is surjective, so $\rho_{E,2}$ is surjective. See \cite{RouseZB} for more discussion about the images of 2-adic Galois representations associated to elliptic curves. Also, the discriminant of the maximal order of the degree 3 field obtained by adjoining one root of the 2-division polynomial is $83$.

The polynomial $T_4(x)=x^{12} + 54bx^{10} + (132a^3 + 891b^2)x^8 + (432a^3b + 2916b^3)x^6 + (-528a^6 - 7128a^3b^2 - 24057b^4)x^4 + (864a^6b + 11664a^3b^3 + 39366b^5)x^2 - 64a^9 - 1296a^6b^2 - 8748a^3b^4 - 19683b^6$ given in \cite{note} is $x^{12} - 2^23^6199^1x^{10} + 2^83^{13}11^183^1x^8 - 2^{11}3^{18}83^1199^1x^6 -  2^{16}3^{25}11^183^2x^4 - 2^{18}3^{30}83^2199^1x^2 - 2^{24}3^{36}83^3$ and its Galois group is $\mathrm{GL}_2(\Z/4\Z)$ in this case.

Fix the point $P=(15, -108) \in E(\Q)$. Then the $[2]^{-1}P$ polynomials are given by $f_x(x)=x^4 - 60x^3 - 2538x^2 + 9828x + 2255121$ and $f_{x,y}(x,y)=x^6 + 6345x^4 + 864x^3y - 214920x^3 - 8051805x^2 + 1096416xy + 54546696x - 9284544y - 2967360237$. Taking $f_y=\mathrm{Res}(f_x, f_{x,y},x) = f_y=y^4 +864y^3 +34992y^2 -11292058368$, see \ref{ypoly}, we see $\mathrm{Gal}(f_x(x)/\Q)=\mathrm{Gal}(f_y(y)/\Q)=\mathrm{Gal}( f_x(x) f_y(x) ) =S_4 \simeq \mathrm{AGL}_2(\Z/2/Z)$ and the discriminant of the maximal order of the degree 4 field obtained by adjoining one root of the this polynomial is $2^483$.

Similarly, we can define  $f_x=f_{E,[4]^{-1}P,x}(x)$, $f_{xy}=f_{E,[4]^{-1}P,x,y}(x,y)$, and $f_y=f_{E,[4]^{-1}P,y}(y)$, then $\mathrm{Gal}(f_x(x) /\Q)=\mathrm{Gal}(f_y(y) /\Q)=\mathrm{Gal}(f_x(x) f_y(x)/\Q)=\mathrm{AGL}_2(\Z/4\Z)$.

\section{Geometry of a quaternion origami}

\begin{definition}
An origami is a pair $(C,f)$ where $C$ is a curve and $f: C \rightarrow E$ is a map branched at at most one point. 
\end{definition}

\begin{definition}
A deck transformation or automorphism of a cover $f: C \rightarrow E$ is a homeomorphism $g: C \rightarrow C$ such that $f \circ g = f$. 
\end{definition}

\subsection{Grothendieck background}

Let $X=E \setminus \left\{ \mathcal{O} \right\}$.

There is an exact sequence of fundamental groups (for example, see \cite{NTM})
$$ 1 \rightarrow \pi_1(X_{\overline{\Q}})  \rightarrow \pi_1(X) \rightarrow G_{\Q} \rightarrow 1$$

\noindent which group-theoretically gives rise to a representation $$\rho_X: G_{\Q} \rightarrow \mathrm{Aut}(\pi_1(X_{\overline{\Q}})).$$ Comparison theorems allow determination of the geometric fundamental group, $\pi_1(X_{\overline{\Q}})$, by considering the variety $X$ over $\C$ \cite{sga1}.

In the case of a once-punctured elliptic curve, the once-punctured torus can be deformed to a wedge sum of 2 circles and its fundamental group is a free group on 2 generators. The \'{e}tale fundamental group of $X$ is a free pro-$\ell$ group on 2 generators. Therefore, any two-generated $\ell$-group will arise as the group of deck transformations of some origami curve over $E$. For example, the quaternion group of order 8 is the smallest non-abelian example of a 2-group generated by 2 generators.

\subsection{Teichm\"{u}ller theory}

This section is due to work by Herrlich and Schmith{\"u}sen. See for example, especially, \cite{HerrlichS} on ``An extraordinary origami curve". Also, see \cite{herrlich} and \cite{HS}. 

Consider the ramified map $ f:C \rightarrow E \rightarrow E$

Letting $X=E \setminus \left\{ \mathcal{O} \right\}$, $Y=E \setminus \left\{ E[2] \right\}$, $Z=C \setminus \left\{ f^{-1}(\mathcal{O}) \right\}$, so that 

\[\begin{tikzcd}
	Z \arrow[r, "\Z/2\Z"] \arrow[rd, "Q_8"]
	& Y \arrow[d, "V_4" ] \\
	 & X
\end{tikzcd}\]

is the same as the origami with the branch points removed.

Consider the elements of the group $Q_8=\left\{ \pm 1, \pm i, \pm j, \pm k \right\}$. The defining relations are $i^2=j^2=k^2=-1$, $ij=-ji=k$, $(-1)^2=1$. Thus, $ik=-ki=-j$ and $jk=-kj=i$. Draw 8 squares, one labeled with each element of the group. Glue the squares horizontally (respectively vertically) so that the right neighbor of the square labeled $g$ has label $g \cdot i$ and its top neighbor has label $g \cdot j$.

\begin{figure}[!hb] 
\includegraphics[width=80mm]{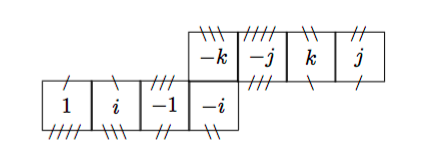}
  \end{figure}
  
  or 
  
  \begin{figure}[!hb] 
\includegraphics[width=50mm]{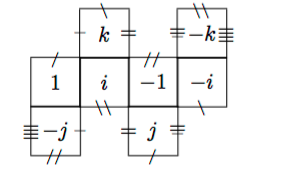}
  \end{figure}

The complex has 8 faces, 16 edges, and 4 vertices. Using the Euler characteristic formula, $2-2g=v-e+f$, the genus of the curve $C$ is 3. Next, applying the Riemann-Hurwitz formula,

$$ f: C \rightarrow E$$
Then,
$$ 2 g(C) -2=\mathrm{deg}(f) (2g(E)-2)+\displaystyle \sum_{P \in E} (e_P-1),$$ 

with $g(C)=3$, $g(E)=1$, and $\mathrm{deg}(f)=8$. This gives that there are 4 points ramified in $C$, each with ramification degree 2. 

In fact, Herrlich and Schmith{\"u}esen \cite{HerrlichS} give that the map $C \rightarrow E$ is given by $(x,y)\mapsto (x,y^2)$ and $C: y^4=x^3+Ax+B$. This is an example of a superelliptic curve. Let $E_{\lambda}: y^2=x(x-1)(x-\lambda)$ and $W_{\lambda}: y^4=x(x-1)(x-\lambda)$. Then, Herrlich and Schmith{\"u}esen show \cite[Proposition 7]{HerrlichS} that for all $\lambda \in \mathbb{P}^1 \backslash \left\{ 0,1, \infty \right\}$, the Jacobian of this genus 3 curve, $\mathrm{Jac}(W_{\lambda})$ is isogenous to $E_{\lambda} \times E_{-1} \times E_{-1}$.


Let $P=(z,w) \in E(\Q)$ be a rational point of $E$, distinct from the origin. Then $f^{-1}(P)$ consists of a set of 8 distinct points on $C$. The $x$-coordinates of these points are the $x$-coordinates of the $[2]^{-1}P$ points. The $y$-coordinates are given by square roots of the $y$-coordinates of the $[2]^{-1}P$ points. To find a polynomial for the $[2]^{-1}P$ polynomial in terms of the $y$-coordinates, we can take the resultant of the polynomials $(x^4-2ax^2-8bx+a^2)-z(4(x^3+ax+b))$ and $(x^6+5ax^4+20bx^3-5a^2x^2-4abx-8b^2-a^3)-w(8y(x^3+ax+b))$. Doing so, we find that a division polynomial for the $y$-coordinates of the $[2]^{-1}P$ points is given by 

\begin{multline*}
 y^4w^4 - 8y^3azw^3 - 8y^3bw^3 - 8y^3z^3w^3 +12y^2a^2z^2w^2+ 30y^2abzw^2 + 12y^2az^4w^2 + 18y^2b^2w^2 +18y^2bz^3w^2\\
 - 4a^5z^2 - 8a^4bz - 8a^4z^4 - 4a^3b^2 -8a^3bz^3 - 4a^3z^6 - 27a^2b^2z^2 - 54ab^3z - 54ab^2z^4 -27b^4 - 54b^3z^3 - 27b^2z^6
\end{multline*}

Next, we will reduce each coefficient using the relation $w^2=z^3+az+b$ since $P=(z,w) \in E(\Q)$.

For example, the $y^3$ coefficient is 
$$-8w^3[az+b+z^3]=-8w^3[w^2].$$

The $y^2$ coefficient of the polynomial is

\begin{align*}
& 6w^2[2a^2z^2+5abz+2az^4+3b^2+3bz^3] \\
& =6w^2[2a^2z^2+2abz+3abz+2az^4+3b^2+3bz^3] \\
&=6w^2[2az(z^3+az+b)+3b(z^3+az+b)] \\
&=6w^2[w^2][2az+3b] 
\end{align*}

Using that 
$$(w^2)^2=(z^3+az+b)^2=z^6+2bz^3+b^2+2az^4+2abz+a^2z^2$$

the constant coefficient of the polynomial is 

\begin{align*}
& -4a^3[a^2z^2+2abz+2az^4+b^2+2bz^3+z^6]-27b^2[a^2z^2+2abz+2az^4+b^2+2bz^3+z^6]\\
&=(-4a^3+27b^2)[ w^4]
\end{align*}

Therefore, the division polynomial for the $y$-coordinates of the $[2]^{-1}P$ points can be simplifed to the following:
$$ y^4w^4-8w^5y^3+6w^4[2az+3b]y^2+w^4[-4a^3+27b^2]$$

Then, if $w \ne 0$, we can divide this polynomial by $w^4$ to find:

$$y^4-8wy^3+6(2az+3b)y^2-(4a^3+27b^2)$$.
\label{ypoly}

\noindent \textbf{Example:} Fix E=`83a1' in Cremona notation and point $P=(15, -108) \in E(\Q)$. Then the $[2]^{-1}P$ polynomials are given by $f_x=x^4 - 60x^3 - 2538x^2 + 9828x + 2255121$ and $f_{xy} = x^6 + 6345x^4 + 864x^3y - 214920x^3 - 8051805x^2 + 1096416xy + 54546696x - 9284544y - 2967360237$, and $f_y(x)=x^4 + 864x^3 + 34992x^2 - 11292058368$, we see that $\mathrm{Gal}(f_y(x)/\Q) = S_4 \simeq \mathrm{AGL}_2(\Z/2/Z)$ and the discriminant of the maximal order of the degree 4 field obtained by adjoining one root of $f_y$ to $\Q$ is $2^483$. In fact, this is the same field as adjoining the roots of $f_x$ to $\Q$.

To take square roots of the roots of this polynomial, we can plug in $y^2$ for $y$ and find that
$$ f_{E,Q_8, P} =y^8-8wy^6+6(2az+3b)y^4-(4a^3+27b^2).$$

Note that $-P=(z,-w)$ is also a rational point on the curve. We have that 

$$f_{E,Q_8,-P}=y^8+8wy^6+6(2az+3b)y^4-(4a^3+27b^2).$$ 

Another way to find both $f_{E, Q_8, P}$ and $f_{E, Q_8, -P}$ is to first take the resultant of 

$$(x^4-2ax^2-8bx+a^2)-z(4(x^3+ax+b))$$ 

and 

$$y^2-(x^3+ax+b)$$ eliminating $x$. This yields

\begin{align*}
r & =y^8 - 40y^6az - 28y^6b - 64y^6z^3 - 8y^4a^3 +144y^4a^2z^2\\
& +432y^4abz + 270y^4b^2 - 96y^2a^4z - 144y^2a^3b - 648y^2ab^2z - 972y^2b^3 + 16a^6 + 216a^3b^2 + 729b^4
\end{align*}

 Let $s$ denote the polynomial obtained by evaluating $r$ at $y^2$.
        
Note that $f_{E,Q_8,P} f_{E,Q_8,-P}=(y^8-8wy^6+6(2az+3b)y^4-(4a^3+27b^2))(y^8+8wy^6+6(2az+3b)y^4-(4a^3+27b^2))$. Then, $s-f_{E,Q_8,P} f_{E,Q_8,-P}=-64y^{12}az - 64y^{12}b - 64y^{12}z^3 + 64y^{12}w^2=64y^{12}[ -az-b-z^3+w^2]$. Since $(z,w) \in E(\mathbb{Q})$, $w^2-z^3-az-b=0$, so $s=f_{E,Q_8,P} f_{E, Q_8,-P}$.

For our running example: take $E=`83a1'$ and $P=(15,-108)$. Then,

$$f_{E,Q_8,P}=y^8 + 864y^6 + 34992y^2 - 11292058368$$

 and 
 
 $$f_{E,Q_8,-P}=y^8 - 864y^6 + 34992y^2 - 11292058368.$$ 
 
 Here, 
 
 $$\mathrm{Gal}(f_{E, Q_8,P}/\Q)=\mathrm{Hol}(Q_8)$$ 
 
 and 
 
 $$\mathrm{Gal}( f_{E, Q_8,-P}/\Q)=\mathrm{Hol}(Q_8).$$ 

where the holomorph of a group is the semi-direct product $\mathrm{Hol}(G)=G \rtimes \mathrm{Aut}(G)$ where 
the action $\phi: \mathrm{Aut}(G) \rightarrow \mathrm{Aut}(G)$ is the identity \cite{Hall}.

There is exact sequence

$$ 1 \rightarrow G \rightarrow \mathrm{Hol}(G) \rightarrow \mathrm{Aut}(G) \rightarrow 1.$$
 \label{holomorph}

In particular, there is an exact sequence

$$ 1 \rightarrow Q_8 \rightarrow \mathrm{Hol}(Q_8) \rightarrow \mathrm{Aut}(Q_8) \simeq S_4 \rightarrow 1.$$

\begin{remark} 
 The group $\mathrm{Hol}(Q_8)$ is $\mathrm{SmallGroup}(192,1494)$ in Magma notation \cite{magma}. 
 \end{remark}
 
 \begin{remark} In the special case $E=`83a1'$ and $P=(15,-108)$, the $\mathrm{Gal}( f_{E,Q_8,P} f_{E, Q_8,-P}/\Q) =\mathrm{SmallGroup}(384,20090)$ in Magma notation \cite{magma}.
  \end{remark}

\begin{theorem} Let $E:y^2=x^3+ax+b$ be an elliptic curve over $\Q$. Fix a point $P=(z,w) \in E(\Q) \setminus \left\{ \mathcal{O} \right\}$. Suppose that the representation $\rho_{E, [2]^{-1}P}$ surjects onto $S_4$. Let $L_P$ denote the splitting field of the polynomial
$$f_P(x)=x^8-8wx^6+6(2az+3b)x^4-(4a^3+27b^2).$$
Then $\mathrm{Gal}(L_P/\Q) \hookrightarrow \mathrm{Hol}(Q_8)$.
\end{theorem}
\label{Galoisgroup}

\begin{proof}
Let $\Delta=-16(4a^3+277b^2)$. Let $r(x)=x^4-8wx^3+6(2az+3b)x^2-(4a^3+27b^2)$ be the degree 4 polynomial whose roots are the $y$-coordinates of $[2]^{-1}P$ and $s(x)=x^2$. Then, $f_P(x)=r(s(x))$.

We have that $\mathrm{Gal}( \Q([2]^{-1}P)/\Q) \leq \mathrm{AGL}_2(\Z/2\Z) \simeq S_4$. Since this representation is surjective, the resolvent cubic of $r(x)$ has Galois group, $S_3$ over $\Q$. This is the same as the mod-2 representation of $E$ by studying affine general linear representations. Therefore, $\overline{\rho}_{E,2}$ is surjective. Thus, $E$ has no rational points of order 2, so by \cite[Corollary 5.3(3)]{BK}, neither of $\pm \Delta$ is a square.

By Odoni \cite{Odoni}[Lemma 4.1], $\mathrm{Gal}(r(s(x))/\Q ) \hookrightarrow S_2 \wr S_4 =\mathrm{SmallGroup}(384,5604)$. 

By Proposition 9.7.2 in \cite{Goodman}, the Galois group of the polynomial over $\Q$ is contained in $A_8$, the alternating group of order 8, if and only the discriminant of the polynomial has a square root in $\Q$. 

Suppose $A=a_nx^n+\ldots +a_0$. Then,
$\mathrm{Disc}_x(A)=\frac{(-1)^{\left(\frac{n(n-1)}{2}\right)}}{a_n}\mathrm{Res}_x(A,A')$, where $A'$ is the derivative of $A$. 
Taking $A=x^8-8wx^6+6(2az+3b)x^4-(4a^3+27b^2)$, then $A'=8x^7-48wx^5+24(2az+3b)x^3$ and the discriminant is 
$-2^{32}(4a^3 + 27b^2)^3
(81z^4a^4 - 108z^3w^2a^3 + 486z^3a^3b - 486z^2w^2a^2b + 18z^2a^5 + 1215z^2a^2b^2 - 108zw^2a^4 - 
    1458zw^2ab^2 + 54za^4b + 1458zab^3 + 108w^4a^3 + 729w^4b^2 - 162w^2a^3b - 1458w^2b^3 + a^6 + 54a^3b^2 + 729b^4)^2$. Because of the $(4a^3 + 27b^2)^3$ factor and since $\pm \Delta$ is not a square, the discriminant of the polynomial is not a square and therefore the Galois group is not contained in $A_8$.

Next, we consider the transitive subgroups of $S_8$, which are also subgroups of $S_2 \wr S_4$, which are not subgroups of $A_8$, and that have an $S_4$-quotient. Except for $S_2 \wr S_4$, all such subgroups are subgroups of $\mathrm{Hol}(Q_8)$, so it remains to rule out $S_2 \wr S_4$ as a possible Galois group.

To do so, we will follow \cite[p.37]{Mattman}. Consider a degree 8 polynomial $f(x)=\displaystyle \prod_{i=1}^8 (x - \alpha_i)$ in the form $r(x^2)$, where $r(x)=x^4+c_3x^3+c_2x^2+c_1x+c_0=(x-r_1)(x-r_2)(x-r_3)(x-r_4)$. Suppose that $\alpha_1 \alpha_8=r_1$, $\alpha_2 \alpha_7=r_2$, $\alpha_3 \alpha_6=r_3$, and $\alpha_4 \alpha_5=r_4$. To distinguish between Galois groups $\mathrm{Hol}(Q_8)$ and $S_2 \wr S_4$, we first construct the polynomial $g$ as the 2-set resolvent of the degree 4 factor of the 2-set resolvent of $f$. Next, the 3-set resolvent of $g$ has an irreducible factor $h$ of degree 12. If $\mathrm{Gal}(f/\Q)=\mathrm{Hol}(Q_8)$ (respectively $S_2 \wr S_4$), then $h$ factors (respectively is irreducible over $k(\sqrt{D})$, where $D$ is the discriminant of $f$.

The degree 4 factor of the 2-set (product) resolvent of $f$ is a degree 4 polynomial with roots $\alpha_1\alpha_8$, $\alpha_2\alpha_7$, $\alpha_3\alpha_6$, $\alpha_4\alpha_5$. This polynomial is $\displaystyle \prod_{i=1}^4 (x - r_i) =r(x) = x^4-8wx^3+6(2az+3b)x^2-(4a^3+27b^2)$, the $[2]^{-1}P$ polynomial. The polynomial $k(x)$ is the 2-set resolvent of $r(x)$, so $k(x)=(x-r_1r_2)(x-r_3r_4)(x-r_1r_3)(x-r_2r_4)(x-r_1r_4)(x-r_2r_3)$. Let $s_1=r_1r_2$, $s_2=r_3r_4$, $s_3=r_1r_3$, $s_4=r_2r_4$, $s_5=r_1r_4$, $s_6=r_2r_3$ be the roots of $k(x)$.

In general, for a degree 4 polynomial $r(x)=(x-r_1)(x-r_2)(x-r_3)(x-r_4)=x^4+c_3x^3+c_2x^2+c_1x+c_0$ defined over a number field $k$, the degree 6 2-set (product) resolvent polynomial is given by $k(x)=x^6-c_2x^5+(-c_1c_3-c_0)x^4+(-c_0(c_3^2-2c_2)-c_1^2)x^3+ c_0(-c_1c_3-c_0)x^2-c_0^2c_2x+c_0^3$. In the case that $r(x)=x^4-8wx^3+6(2az+3b)x^2-(4a^3+27b^2)$, the degree 6 polynomial $k(x)$ is given by $k(x)=x^6-(12az + 18b)x^5+(4a^3+27b^2)x^4+(4a^3+27b^2)(64w^2-24az-36b)x^3-(4a^3+27b^2)^2x^2-(4a^3+27b^2)^2(12az + 18b)x+(-4a^3 - 27b^2)^3$. 

For example, for elliptic curve `83a1' with $P=(z,w)=(15,-108)$, $f_P(x)=x^8+864x^6+2^43^7x^4-2^83^{12}83^1$, $r(x)=x^4+864x^3+2^43^7x^2-2^83^{12}83^1$, $k(x)=x^6 - 2^43^7x^5 + 2^83^{12}83^1x^4 + 2^{13}3^{18}29^183^1x^3 -2^{16}7^119^1103^12281^12579^15209^1x^2 -2^{20}3^{31}83^2x -2^{24}3^{36}83^3$.

Next, the 3-set (sum) resolvent of $k(x)$ has an irreducible degree 12 factor $h(x)=(x-(s_1+s_2+s_3))(x-(s_1+s_2+s_4))(x-(s_1+s_2+s_5))(x-(s_1+s_2+s_6))(x-(s_1+s_3+s_4))(x-(s_1+s_5+s_6))(x-(s_2+s_3+s_4))(x-(s_2+s_5+s_6))(x-(s_3+s_4+s_5))(x-(s_3+s_4+s_6))(x-(s_3+s_5+s_6))(x-(s_4+s_5+s_6))$.

Using Vieta's formulas, gives that the degree 12 factor of the 3-set (sum) resolvent of $k(x)$ is given in \hyperref[appendixA]{Appendix A}.

 The degree 12 polynomial $h(x)$ factors as $(x^6+(-36az-54b)x^5+(20a^3 + 432a^2z^2 + 1296abz + 1107b^2)x^4+(-480a^4z - 720a^3b - 1728a^3z^3 - 7776a^2bz^2 - 14904ab^2z - 10692b^3)x^3 +(240a^6 + 4608a^5z^2 + 13824a^4bz - 3072a^4zw^2 + 13608a^3b^2 - 4608a^3bw^2 + 
    31104a^2b^2z^2 + 93312ab^3z - 20736ab^2zw^2 + 80919b^4 - 31104b^3w^2)x^2+v_1x+v_3)(x^6+(-36az-54b)x^5+(20a^3 + 432a^2z^2 + 1296abz + 1107b^2)x^4+(-480a^4z - 720a^3b - 1728a^3z^3 - 7776a^2bz^2 - 14904ab^2z - 10692b^3)x^3+(240a^6 + 4608a^5z^2 + 13824a^4bz - 3072a^4zw^2 + 13608a^3b^2 - 4608a^3bw^2 + 
    31104a^2b^2z^2 + 93312ab^3z - 20736ab^2zw^2 + 80919b^4 - 31104b^3w^2)x^2+v_2x+v_4)$, where $v_1, v_2, v_3, v_4$ are as defined in  \hyperref[appendixA]{Appendix A}.
    
 Therefore, if $v_1, v_2, v_3, v_4 \in \mathbb{Q}(\sqrt{D})$ where $D=\mathrm{Discriminant}(f)$, then $h(x)$ is reducible over $\mathbb{Q}(\sqrt{D})$.
 
 For $f(x)=f_{E,P}(x)=x^8-8wx^6+6(2az+3b)x^4-(4a^3+27b^2)$, $D=-2^{32} (4a^3+27b^2)^3(a^6 + 18a^5z^2 + 54a^4bz + 81a^4z^4 - 108a^4zw^2 + 54a^3b^2 + 486a^3bz^3 - 162a^3bw^2 - 108a^3z^3w^2 + 108a^3w^4 + 1215a^2b^2z^2 - 486a^2bz^2w^2 + 1458ab^3z - 1458ab^2zw^2 + 729b^4 - 1458b^3w^2 + 729b^2w^4)^2$, so $\Q(\sqrt{D})=\Q(\sqrt{-(4a^3+27b^2)})$.
 
We will next consider the discriminants $d_1$ and $d_2$ of the 2 quadratic polynomials that define the roots $v_i$. Then, $d_1=2^{14}(4a^3+27b^2)^3(a^6 + 18a^5z^2 + 54a^4bz + 81a^4z^4 - 108a^4zw^2 + 54a^3b^2 + 486a^3bz^3 - 162a^3bw^2 - 108a^3z^3w^2 + 108a^3w^4 + 1215a^2b^2z^2 - 486a^2bz^2w^2 + 1458ab^3z - 1458ab^2zw^2 + 729b^4 - 1458b^3w^2 + 729b^2w^4)$ and $d_2=-2^{14}3^2(2az+3b)^2(4a^3+27b^2)^3(a^6 + 18a^5z^2 + 54a^4bz + 81a^4z^4 - 108a^4zw^2 + 54a^3b^2 + 486a^3bz^3 - 162a^3bw^2 - 108a^3z^3w^2 + 108a^3w^4 + 1215a^2b^2z^2 - 486a^2bz^2w^2 + 1458ab^3z - 1458ab^2zw^2 + 729b^4 - 1458b^3w^2 + 729b^2w^4)$.

Let
$q=(a^6 + 18a^5z^2 + 54a^4bz + 81a^4z^4 - 108a^4zw^2 + 54a^3b^2 + 486a^3bz^3 - 162a^3bw^2 - 108a^3z^3w^2 + 108a^3w^4 + 1215a^2b^2z^2 - 486a^2bz^2w^2 + 1458ab^3z - 1458ab^2zw^2 + 729b^4 - 1458b^3w^2 + 729b^2w^4)$ and let $u=(27bz^3-9a^2z^2-a^3)^2$.

Then, $q-u=3^2(z^3+az+b-w^2)(2az+2b+2z^3-2w^2)=0$, since $z^3+az+b=w^2$.

Therefore, using the relation that $P=(z,w)$ is a point on the elliptic curve, $d_1$ and $d_2$ simplify to $d_1=-2^{14}(4a^3+27b^2)^3(27bz^3-9a^2z^2-a^3)^2$
and $d_2=-2^{14}3^2(2az+3b)^2(4a^3+27b^2)^3(27bz^3-9a^2z^2-a^3)^2$. Therefore, $p_1$ factors over $\Q(\sqrt{d_1})=\Q(\sqrt{D})$ and $p_2$ factors over $\Q(\sqrt{d_2})=\Q(\sqrt{D})$. Therefore, $h(x)$ factors over $\Q(\sqrt{D})$, which implies that $\mathrm{Gal}(f/\Q) \leq \mathrm{Hol}(Q_8)$.
 
\end{proof}

\begin{remark} Consider the embedding $Q_8 \hookrightarrow S_8$ $i \mapsto (1,3,2,4)(5,7,6,8)$ and $j \mapsto (1,5,2,6)(3,8,4,7)$. Then the subgroup of $S_8$ generated by $i$ and $j$ is isomorphic to $Q_8$ and the normalizer of this group in $S_8$ is isomorphic to the holomorph, $\mathrm{Hol}(Q_8)$.
\end{remark}

For example, for elliptic curve `83a1' with $P=(z,w)=(15,-108)$, $f(x)=f_{E,P}(x)=x^8+864x^6+2^43^7x^4-2^83^{12}83^1$ and $r(x)=x^4+864x^3+2^43^7x^2-2^83^{12}83^1$. The discriminant $D$ of $f(x)$  is $2^{72}3^{84}83^3739^4$. For the polynomial $f(x)$ there are the associated resolvent polynomials, $k(x)=x^6 - 2x^5 - 4x^4 - 16x^3 + 16x^2 - 32x - 64$ and $h(x)=x^{12} - 2^53^8x^{11} + 2^83^{12}5^1193^1x^{10} - 2^{13}3^{19}5^1433^1x^9 +
    2^{18}3^{24}5^17^143^171^1x^8 - 2^{21}3^{31}740603^1x^7 +
    2^{24}3^{36}257^1159337^1x^6 -
    2^{29}3^{44}5^17^161^173^183^1x^5 -
    2^{32}3^{48}5^283^1103^162003^1x^4 +
    2^{38}3^{55}5^111^137^183^2491^1x^3 + \\
    2^{42}3^{60}7^183^2731921^1x^2 -
   2^{48}3^{67}7^183^34153^1x +
    2^{52}3^{72}83^3227^1340633^1$.

The polynomial $h(x)$ is irreducible over $\mathbb{Q}$. Over, $\mathbb{Q}(\sqrt{D})$, $h(x)$ factors as $(x^6 - 2^43^8x^5 + 2^93^{12}13^117^1x^4 - 2^{12}3^{19}839^1x^3 +
        2^{16}3^{26}5^383^1x^2 + 1/(2^{14}3^{12}739^1)(-\sqrt{D} -
        2^{37}3^{46}5^183^1739^1)x + 1/(2^{11}3^5739^1)(\sqrt{D} -
       2^{36}3^{41}5^183^2383^1739^1))(x^6 - 2^43^8x^5 + 2^93^{12}13^117^1x^4 - 2^{12}3^{19}839^1x^3 +
        2^{16}3^{26}5^383^1x^2x^2 + 1/(2^{14}3^{12}739^1)(\sqrt{D} -
        2^{37}3^{46}5^183^1739^1)x + 1/(2^{11}3^5739^1)(-\sqrt{D} -
        2^{36}3^{41}5^183^2383^1739^1))$
        
In this example, $D=2^{72}3^{84}83^3739^4$, , $p_1=x^2 + 2^{24}3^{34}5^183^1x + 2^{44}3^{60}83^245984143^1$, $p_2=x^2 + 2^{26}3^{36}5^183^2383^1x + 
    2^{52}3^{72}83^3227^1340633^1$ $d_1=2^{46}3^{60}83^3739^2$, $d_2= 2^{52}3^{74}83^3739^ 2$.

\textbf{Non-example:} Consider the degree 8 polynomial $f(x)=x^8 - 2x^6 + 2x^4 + 4x^2 - 4$. We can compute that $D=\mathrm{Discriminant}(f)=-2^{26}83^2$. In this case, $k(x)=x^6 - 2x^5 - 4x^4 - 16x^3 + 16x^2 - 32x - 64$ and $h(x)=x^{12} - 12x^{11} + 68x^{10} - 240x^9 + 832x^8 - 3008x^7 + 7104x^6 - 8192x^5 +14592x^4 - 46080x^3 + 239616x^2 - 385024x + 671744$. This degree 12 polynomial $h(x)$ does not factor over $\mathbb{Q}(\sqrt{D})$. By computation, we also find that $\mathrm{Gal}(f(x)/\Q) \simeq S_2 \wr S_4$. 

In this case, $D=2^{16}7^171^110711^1$, $\mathrm{Discriminant}(r(x))=2^883$, $p_1=x^2 + 512x + 405504$, $p_2=x^2 + 1152x + 671744$, $d_1=d_2=-2^{14}83$.

\begin{theorem}
Let $G=\mathrm{Hol}(Q_8)$. Then, $G$ has 3 $S_4$-quotients. Suppose that $\mathrm{Gal}( f_{E,Q_8,P}(x)/\Q )=G$. Then, the 3 subfields with $S_4$ Galois group over $\Q$ are the splitting fields of the following polynomials:
  $$h_1=x^4-512dx^2 +2^{15}dw^2x  + 2^{16}d(d+w^2(12az-36b)) +  2^{18}d(27bz^3 - 9a^2z^2 - a^3)$$
  $$h_2=x^4 -512dx^2 +2^{15}dw^2x  + 2^{16}d(d+w^2(12az-36b))$$
  $$ h_3=x^4-8wx^3+6(2az+3b)x^2-d$$
  where $d=4a^3+27b^2$.
\end{theorem}

\begin{proof}
See the note about the procedure to find defining polynomials for quotients in a similar fashion to Adelmann \cite{note} \cite{Adelmann}.
\end{proof}

Let $\Delta=-16(4a^3+27b^2)$ and $g=x^4-4\Delta x-12 a \Delta$, the polynomial that cuts out the unique $S_4$-quotient of the $\mathrm{GL}_2(\Z /4\Z)$-extension of $\Q$ given by $\Q(E[4])/\Q$ when $\overline{\rho}_{E,4}$ is surjective. See Adelmann \cite{Adelmann} for more detail about the 4-division polynomial and about $g(x)$.

\begin{theorem}[ Main Theorem ]
Let $k_g/\Q$ be the field extension given by adjoining a root of $g=x^4-4\Delta x-12 a \Delta$ to $\Q$. Let $k_1/\Q$ be the field extension given by adjoining a root of $h_1$ to $\Q$. 

Let $\alpha$ be a root of $k_g$. Take $\beta = \frac{-1}{3^2b}(\alpha^3-4a\alpha^2+(8a^2-72bz)\alpha +2^4 3 d)$. This map gives an isomorphism between $k_g$ and $k_1$. Note that the isomorphism depends on the point $P$ as does $k_1$, whereas $k_g$ is independent of the point.
\end{theorem}

\begin{proof} Consider $3^8b^4h_1(\beta)$. It is a degree 12 polynomial in $\alpha$ with coefficients in $\Z[a,b,z,w]$. See \hyperref[appendixB]{Appendix B} for the polynomial. After, reducing this polynomial using that $\alpha^4=4\Delta \alpha+12 a \Delta$, we get that $3^8b^4h_1(\beta)=c_3 \alpha^3+c_2\alpha^2+c_1 \alpha+c_0$ with $c_i \in \Z[a,b,z,w]$. See \hyperref[appendixB]{Appendix B} for these coefficients. After further using that $P=(z,w)$ is a point on $E$, so $z^3+az+b-w^2=0$, each of these coefficients reduces to zero and so this is the zero polynomial. Therefore, $\beta$ satisfies the polynomial $3^8b^4 h_1(x)$ and the map is invertible, so the map is an isomorphism.
\end{proof}

Specializing to the case of the elliptic curve `83a1' with point $P=(15,-108)$, we get that 
\begin{align*}
h_1&=x^4 - 2^{17}3^{12}83^1x^2 + 2^{27}3^{18}83^1x - 2^{32}3^{24}7^283^1\\
h_2&=x^4 - 2^{17}3^{12}83^1x^2 +  2^{27}3^{18}83^1x + 2^{30}3^{25}83^1181^1\\
h_3&=x^4 + 2^53^3x^3 - 2^43^7x^2 - 2^83^{12}83^1\\
g &=x^4 + 2^{14}3^{12}83^1x + 2^{14}3^{16}47^183^1\\ 
\beta&=1/(2^13^5199^1)(\alpha^3 - 2^23^347^1\alpha^2 + 2^63^889^1\alpha + 2^{12}3^{13}83^1)
\end{align*}

\newpage

\begin{center} \textbf{Appendix A} \end{center}
\label{sec: appendixA}

For $f(x)=r(x^2)$ of degree 8 and discriminant $D$, \cite{Mattman} gives a degree 12 resolvent polynomial $h(x)$ irreducible over $\Q$ such that whether the reducibility/irreducibility of $h(x)$ over $\Q(\sqrt{D})$ can be used to distinguish between certain possible Galois groups of $f(x)$ over $\Q$. In this appendix, we give this resolvent polynomial, in general and specialized to our context.

Suppose that $f(x)=r(x^2)$, where $r(x)=x^4+c_3x^3+c_2x^2+c_1x+c_0$, in general and we will then specialize to the case that $r(x)=x^4-8wx^3+6(2az+3b)x^2-(4a^3+27b^2)$ where $E: y^2=x^3+ax+b$ and $P=(z,w) \in E(\Q)$.

$h(x)=x^{12}+d_{11}x^{11}+d_{10}x^{10}+d_9x^9+d_8x^8+d_7x^7+d_6x^6+d_5x^5+d_4x^4+d_3x^3+d_2x^2+d_1x+d_0= x^{12} - 6c_2x^{11} + (-10c_0 + 4c_1c_3 + 15c_2^2)x^{10} + (50c_0c_2 - 20c_1c_2c_3 - 20c_2^3)x^9 + (55c_0^2 - 
    34c_0c_1c_3 - 106c_0c_2^2 + 2c_0c_2c_3^2 + 2c_1^2c_2 + 6c_1^2c_3^2 + 40c_1c_2^2c_3 + 15c_2^4)x^8 + 
    (-190c_0^2c_2 + 122c_0c_1c_2c_3 + 118c_0c_2^3 - 6c_0c_2^2c_3^2 - 6c_1^2c_2^2 - 22c_1^2c_2c_3^2 - 
    40c_1c_2^3c_3 - 6c_2^5 + v_1 + v_2)x^7 + (-150c_0^3 + 130c_0^2c_1c_3 + 270c_0^2c_2^2 - 10c_0^2c_2c_3^2 - 
    10c_0c_1^2c_2 - 38c_0c_1^2c_3^2 - 174c_0c_1c_2^2c_3 + 4c_0c_1c_2c_3^3 - 68c_0c_2^4 + 6c_0c_2^3c_3^2 + 
    4c_1^3c_2c_3 + 4c_1^3c_3^3 + 6c_1^2c_2^3 + 30c_1^2c_2^2c_3^2 + 20c_1c_2^4c_3 + c_2^6 - 3c_2v_1 - 3c_2v_2 + 
    v_3 + v_4)x^6 + (300c_0^3c_2 - 260c_0^2c_1c_2c_3 - 190c_0^2c_2^3 + 20c_0^2c_2^2c_3^2 + 20c_0c_1^2c_2^2 + 
    76c_0c_1^2c_2c_3^2 + 118c_0c_1c_2^3c_3 - 8c_0c_1c_2^2c_3^3 + 16c_0c_2^5 - 2c_0c_2^4c_3^2 - 5c_0v_1 - 
    5c_0v_2 - 8c_1^3c_2^2c_3 - 8c_1^3c_2c_3^3 - 2c_1^2c_2^4 - 18c_1^2c_2^3c_3^2 - 4c_1c_2^5c_3 + 2c_1c_3v_1 + 
    2c_1c_3v_2 + 3c_2^2v_1 + 3c_2^2v_2 - 3c_2v_3 - 3c_2v_4)x^5 + (225c_0^4 - 210c_0^3c_1c_3 - 240c_0^3c_2^2 +
    30c_0^3c_2c_3^2 + 30c_0^2c_1^2c_2 + 79c_0^2c_1^2c_3^2 + 172c_0^2c_1c_2^2c_3 - 14c_0^2c_1c_2c_3^3 + 
    64c_0^2c_2^4 - 16c_0^2c_2^3c_3^2 + c_0^2c_2^2c_3^4 - 14c_0c_1^3c_2c_3 - 14c_0c_1^3c_3^3 - 16c_0c_1^2c_2^3 -
    42c_0c_1^2c_2^2c_3^2 + 2c_0c_1^2c_2c_3^4 - 32c_0c_1c_2^4c_3 + 4c_0c_1c_2^3c_3^3 + 10c_0c_2v_1 + 
    10c_0c_2v_2 - 5c_0v_3 - 5c_0v_4 + c_1^4c_2^2 + 2c_1^4c_2c_3^2 + c_1^4c_3^4 + 4c_1^3c_2^3c_3 + 
    4c_1^3c_2^2c_3^3 + 4c_1^2c_2^4c_3^2 - 4c_1c_2c_3v_1 - 4c_1c_2c_3v_2 + 2c_1c_3v_3 + 2c_1c_3v_4 - c_2^3v_1 - 
    c_2^3v_2 + 3c_2^2v_3 + 3c_2^2v_4)x^4 + (15c_0^2v_1 + 15c_0^2v_2 - 7c_0c_1c_3v_1 - 7c_0c_1c_3v_2 - 
    8c_0c_2^2v_1 - 8c_0c_2^2v_2 + c_0c_2c_3^2v_1 + c_0c_2c_3^2v_2 + 10c_0c_2v_3 + 10c_0c_2v_4 + c_1^2c_2v_1 + 
    c_1^2c_2v_2 + c_1^2c_3^2v_1 + c_1^2c_3^2v_2 + 2c_1c_2^2c_3v_1 + 2c_1c_2^2c_3v_2 - 4c_1c_2c_3v_3 - 
    4c_1c_2c_3v_4 - c_2^3v_3 - c_2^3v_4)x^3 + (15c_0^2v_3 + 15c_0^2v_4 - 7c_0c_1c_3v_3 - 7c_0c_1c_3v_4 - 
    8c_0c_2^2v_3 - 8c_0c_2^2v_4 + c_0c_2c_3^2v_3 + c_0c_2c_3^2v_4 + c_1^2c_2v_3 + c_1^2c_2v_4 + c_1^2c_3^2v_3 + 
    c_1^2c_3^2v_4 + 2c_1c_2^2c_3v_3 + 2c_1c_2^2c_3v_4 + v_1v_2)x^2 + (v_1v_4 + v_2v_3)x + v_3v_4$, where $v_1, v_2, v_3, v_4 \in \overline{\mathbb{Q}}$ are defined such that $p_1=x^2-\mathrm{sum}_1x+\mathrm{product}_1=(x-v_1)(x-v_2)$ and $p_2=x^2-\mathrm{sum}_2x+\mathrm{product}_2=(x-v_3)(x-v_4)$.
    
 Let $e_1=r_1+r_2+r_3+r_4$, $e_2=r_1r_2+r_1r_3+r_1r_4+r_2r_3+r_2r_4+r_3r_4$, $e_3=r_1r_2r_3+r_1+r_2r_4+r_1r_3r_4+r_2r_3r_4$, and $e_4=r_1r_2r_3r_4$ be the elementary symmetric polynomials in the roots of $r(x)$.
    
  Then, $\mathrm{product}_2=v_3v_4=e_1^8e_4^4 - 2e_1^7e_2e_3e_4^3 + e_1^6e_2^2e_3^2e_4^2 - 7e_1^6e_2e_4^4 + 4e_1^6e_3^2e_4^3 + 
    14e_1^5e_2^2e_3e_4^3 - 6e_1^5e_2e_3^3e_4^2 + 2e_1^5e_3e_4^4 - 7e_1^4e_2^3e_3^2e_4^2 + 
    2e_1^4e_2^2e_3^4e_4 + 17e_1^4e_2^2e_4^4 - 23e_1^4e_2e_3^2e_4^3 + 6e_1^4e_3^4e_4^2 - 10e_1^4e_4^5 - 
    27e_1^3e_2^3e_3e_4^3 + 29e_1^3e_2^2e_3^3e_4^2 - 6e_1^3e_2e_3^5e_4 + 3e_1^3e_2e_3e_4^4 + 
    4e_1^3e_3^3e_4^3 + e_1^2e_2^5e_4^3 + 12e_1^2e_2^4e_3^2e_4^2 - 7e_1^2e_2^3e_3^4e_4 - 29e_1^2e_2^3e_4^4 + 
    e_1^2e_2^2e_3^6 + 29e_1^2e_2^2e_3^2e_4^3 - 23e_1^2e_2e_3^4e_4^2 + 35e_1^2e_2e_4^5 + 4e_1^2e_3^6e_4 - 
    19e_1^2e_3^2e_4^4 + 13e_1e_2^4e_3e_4^3 - 27e_1e_2^3e_3^3e_4^2 + 14e_1e_2^2e_3^5e_4 + 
    11e_1e_2^2e_3e_4^4 - 2e_1e_2e_3^7 + 3e_1e_2e_3^3e_4^3 + 2e_1e_3^5e_4^2 - 10e_1e_3e_4^5 - 4e_2^6e_4^3 +
    e_2^5e_3^2e_4^2 + 33e_2^4e_4^4 - 29e_2^3e_3^2e_4^3 + 17e_2^2e_3^4e_4^2 - 54e_2^2e_4^5 - 7e_2e_3^6e_4 + 
    35e_2e_3^2e_4^4 + e_3^8 - 10e_3^4e_4^3 + 25e_4^6$.
    
    Also, $\mathrm{sum}_1=v_1+v_2=-2e_1^2e_2^2e_4 - 2e_1^2e_2e_3^2 + 14e_1e_2e_3e_4 + 6e_2^3e_4 - 2e_2^2e_3^2 - 30e_2e_4^2$.
    Also, $\mathrm{sum}_2=v_3+v_4=c_0=-2e_1^4e_4^2 + 2e_1^3e_2e_3e_4 + 7e_1^2e_2e_4^2 - 4e_1^2e_3^2e_4 - 7e_1e_2^2e_3e_4 + 2e_1e_2e_3^3 - 2e_1e_3e_4^2 + 2e_2^2e_4^2 + 7e_2e_3^2e_4 - 2e_3^4 + 10e_4^3$.
    
 Let $g_2=-2e_1^6e_2e_4^3 - 2e_1^6e_3^2e_4^2 - 2e_1^5e_2^2e_3e_4^2 + 2e_1^5e_2e_3^3e_4 + 14e_1^5e_3e_4^3 + 
    e_1^4e_2^4e_4^2 + 6e_1^4e_2^3e_3^2e_4 + e_1^4e_2^2e_3^4 + 23e_1^4e_2^2e_4^3 - 13e_1^4e_2e_3^2e_4^2 - 
    4e_1^4e_3^4e_4 - 3e_1^4e_4^4 - 23e_1^3e_2^3e_3e_4^2 - 25e_1^3e_2^2e_3^3e_4 + 2e_1^3e_2e_3^5 - 
    39e_1^3e_2e_3e_4^3 + 30e_1^3e_3^3e_4^2 - 6e_1^2e_2^5e_4^2 - 18e_1^2e_2^4e_3^2e_4 + 6e_1^2e_2^3e_3^4 - 20e_1^2e_2^3e_4^3 + 171e_1^2e_2^2e_3^2e_4^2 - 13e_1^2e_2e_3^4e_4 - 29e_1^2e_2e_4^4 - 2e_1^2e_3^6 - 
    30e_1^2e_3^2e_4^3 + 102e_1e_2^4e_3e_4^2 - 23e_1e_2^3e_3^3e_4 - 2e_1e_2^2e_3^5 - 213e_1e_2^2e_3e_4^3 - 39e_1e_2e_3^3e_4^2 + 14e_1e_3^5e_4 + 92e_1e_3e_4^4 + 9e_2^6e_4^2 - 6e_2^5e_3^2e_4 + e_2^4e_3^4 - 
    122e_2^4e_4^3 - 20e_2^3e_3^2e_4^2 + 23e_2^2e_3^4e_4 + 303e_2^2e_4^4 - 2e_2e_3^6 - 29e_2e_3^2e_4^3 - 
    3e_3^4e_4^2 - 106e_4^5$.
    
Then, $\mathrm{product}_1=v_1v_2=g_2-(v_3 + v_4)(15c_0^2 - 7c_0c_1c_3 - 8c_0c_2^2 + c_0c_2c_3^2 + c_1^2c_2 + c_1^2c_3^2 + 2c_1c_2^2c_3)$. 
    
Now, specializing to the case $f_P(x)=f(x)=r(x^2)=x^8-8wx^6+6(2az+3b)x^4-(4a^3+27b^2)$, this polynomial evaluates as 

  $h(x)=x^{12} + (-2^33^2az - 2^23^3b)x^{11} + (2^35^1a^3 + 2^43^35^1a^2z^2 + 2^43^45^1abz + 2^13^35^119^1b^2)x^{10} + (-2^53^15^2a^4z - 
    2^43^25^2a^3b - 2^83^35^1a^3z^3 - 2^73^55^1a^2bz^2 - 2^33^45^17^111^1ab^2z - 2^23^55^129^1b^3)x^9 + (2^45^111^1a^6 + 
    2^73^253^1a^5z^2 + 2^73^353^1a^4bz + 2^83^55^1a^4z^4 - 2^{11}3^1a^4zw^2 + 2^33^3691^1a^3b^2 + 2^93^65^1a^3bz^3 - 
    2^{10}3^2a^3bw^2 +2^53^5593^1a^2b^2z^2 + 2^53^6233^1ab^3z - 2^93^4ab^2zw^2 + 3^611^1317^1b^4 - 
    2^83^5b^3w^2)x^8 + (-2^73^15^119^1a^7z - 2^63^25^119^1a^6b - 2^93^359^1a^6z^3 - 2^83^559^1a^5bz^2 - 2^{11}3^6a^5z^5 + 
    2^{13}3^3a^5z^2w^2 - 2^63^413^189^1a^4b^2z - 2^{10}3^75^1a^4bz^4 + 2^{13}3^4a^4bzw^2 - 2^53^5449^1a^3b^3 - 
    2^73^6419^1a^3b^2z^3 + 2^{11}3^5a^3b^2w^2 - 2^63^8179^1a^2b^3z^2 + 2^{11}3^6a^2b^2z^2w^2 - 
    2^33^729^1151^1ab^4z + 2^{11}3^7ab^3zw^2 - 2^23^85^113^119^1b^5 + 2^93^8b^4w^2 + v_1 + v_2)x^7 + (2^73^15^2a^9 + 
    2^93^55^1a^8z^2 + 2^93^65^1a^7bz + 2^{12}3^417^1a^7z^4 - 2^{13}3^15^1a^7zw^2 + 2^53^55^141^1a^6b^2 + 
    2^{13}3^517^1a^6bz^3 - 2^{12}3^25^1a^6bw^2 + 2^{12}3^6a^6z^6 - 2^{15}3^4a^6z^3w^2 + 2^83^7151^1a^5b^2z^2 + 
    2^{12}3^8a^5bz^5 - 2^{14}3^6a^5bz^2w^2 + 2^83^7181^1a^4b^3z + 2^{11}3^731^1a^4b^2z^4 - 
    2^{12}3^459^1a^4b^2zw^2 + 2^33^8929^1a^3b^4 +2^{16}3^8a^3b^3z^3 - 2^{11}3^523^1a^3b^3w^2 - 
    2^{13}3^7a^3b^2z^3w^2 + 2^53^{10}11^137^1a^2b^4z^2 - 2^{12}3^9a^2b^3z^2w^2 + 2^53^{10}389^1ab^5z - 
    2^93^7113^1ab^4zw^2 - 2^23^2az(v_1+v_2)+ 2^13^{10}5^17^171^1b^6 - 2^83^841^1b^5w^2 - 2^13^3b(v_1+v_2) + v_3 
    + v_4)x^6 + (-2^{10}3^25^2a^{10}z - 2^93^35^2a^9b - 2^{11}3^35^119^1a^9z^3 - 2^{10}3^55^119^1a^8bz^2 - 2^{16}3^5a^8z^5 + 
    2^{16}3^25^1a^8z^2w^2 - 2^83^65^143^1a^7b^2z - 2^{15}3^65^1a^7bz^4 + 2^{16}3^35^1a^7bzw^2 + 
    2^{17}3^4a^7z^4w^2 - 2^73^65^153^1a^6b^3 - 2^{10}3^65^223^1a^6b^2z^3 + 2^{14}3^45^1a^6b^2w^2 + 
    2^{18}3^5a^6bz^3w^2 - 2^93^95^117^1a^5b^3z^2 - 2^{14}3^8a^5b^2z^5 + 2^{15}3^523^1a^5b^2z^2w^2 - 
    2^63^95^229^1a^4b^4z - 2^{13}3^95^1a^4b^3z^4 + 2^{15}3^611^1a^4b^3zw^2 + 2^{15}3^7a^4b^2z^4w^2 - 
    2^53^9647^1a^3b^5 - 2^73^95^1211^1a^3b^4z^3 + 2^{16}3^7a^3b^4w^2 + 2^{16}3^8a^3b^3z^3w^2 + 2^25^1a^3(v_1+v_2)
    - 2^63^{11}5^183^1a^2b^5z^2 + 2^{12}3^841^1a^2b^4z^2w^2 + 2^43^3a^2z^2(v_1 + v_2) - 
    2^43^{11}5^1311^1ab^6z + 2^{12}3^917^1ab^5zw^2 + 2^43^4abz(v_1 +  v_2) - 2^23^2az(v_3 +v_4) - 
    2^33^{12}11^137^1b^7 + 2^{10}3^{10}11^1b^6w^2 + 3^341^1b^2(v_1 + v_2) - 2^13^3b(v_3 +v_4)x^5 + (2^83^25^2a^{12} + 
    2^{14}3^35^1a^{11}z^2 + 2^{14}3^45^1a^{10}bz + 2^{18}3^4a^{10}z^4 - 2^{15}3^25^1a^{10}zw^2 + 2^83^65^17^1a^9b^2 + 
    2^{19}3^54a^9bz^3 - 2^{14}3^35^1a^9bw^2 - 2^{20}3^3a^9z^3w^2 + 2^{12}3^737^1a^8b^2z^2 - 
    2^{19}3^5a^8bz^2w^2 + 2^{20}3^2a^8z^2w^4 + 2^{12}3^747^1a^7b^3z + 2^{17}3^7a^7b^2z^4 - 
    2^{13}3^637^1a^7b^2zw^2 + 2^{20}3^3a^7bzw^4 + 2^53^811^197^1a^6b^4 + 2^{18}3^8a^6b^3z^3 - 
    2^{12}3^647^1a^6b^3w^2 - 2^{19}3^6a^6b^2z^3w^2 + 2^{18}3^4a^6b^2w^4 + 2^{10}3^{11}23^1a^5b^4z^2 - 
    2^{18}3^8a^5b^3z^2w^2 + 2^{19}3^5a^5b^2z^2w^4 + 2^{10}3^{10}79^1a^4b^5z + 2^{14}3^{10}a^4b^4z^4 - 2^{11}3^{10}23^1a^4b^4zw^2 + 2^{19}3^6a^4b^3zw^4 - 2^53^15^1a^4z(v_1 +v_2) + 2^43^{12}7^137^1a^3b^6 + 2^{15}3^{11}a^3b^5z^3 - 2^{10}3^979^1a^3b^5w^2 - 2^{16}3^9a^3b^4z^3w^2 + 2^{17}3^7a^3b^4w^4 -  2^43^25^1a^3b(v_1+v_2) - 2^63^3a^3z^3(v_1+v_2) + 2^25^1a^3(v_3 +v_4)+  2^83^{12}101^1a^2b^6z^2 - 2^{15}3^{11}a^2b^5z^2w^2 + 2^{16}3^8a^2b^4z^2w^4 - 2^53^5a^2bz^2(v_1 +v_2) + 2^43^3a^2z^2(v_3 + v_4) + 2^83^{13}37^1ab^7z - 2^93^{11}101^1ab^6zw^2 +  2^{16}3^9ab^5zw^4 - 2^33^423^1ab^2z(v_1 +v_2) + 2^43^4abz(v_3 +v_4)+ 
    3^{14}37^21b^8 - 2^83^{12}37^1b^7w^2 + 2^{14}3^{10}b^6w^4 - 2^23^511^1b^3(v_1+v_2) + 3^341^1b^2(v_3 + v_4))x^4 + (2^43^15^1a^6(v_1+v_2) + 2^93^2a^5z^2(v_1 + v_2) + 2^93^3a^4bz(v_1 + v_2) - 2^{10}3^1a^4zw^2(v_1+v_2)  - 2^53^15^1a^4z(v_3 +v_4) + 2^33^57^1a^3b^2(v_1 +v_2)  - 2^93^2a^3bw^2(v_1+v_2)  - 2^43^25^1a^3b(v_3+v_4) - 2^63^3a^3z^3(v_3+v_4)
    + 2^73^5a^2b^2z^2(v_1 +v_2) - 2^53^5a^2bz^2(v_3+v_4)  + 
    2^73^6ab^3z(v_1 + v_2) - 2^83^4ab^2zw^2(v_1+v_2) - 2^33^423^1ab^2z(v_3 +v_4) 
    + 3^737^1b^4(v_1 + v_2) - 2^73^5b^3w^2(v_1 +v_2) - 2^23^511^1b^3(v_3 + v_4))x^3 + (2^43^15^1a^6(v_3 + v_4) + 2^93^2a^5z^2(v_3 + v_4) + 2^93^3a^4bz(v_3 + v_4) - 2^{10}3^1a^4zw^2(v_3+v_4) + 2^33^57^1a^3b^2(v_3 + v_4) - 
    2^93^2a^3bw^2(v_3 +v_4) + 2^73^5a^2b^2z^2(v_3 + v_4) + 2^73^6ab^3z(v_3 + v_4) - 2^83^4ab^2zw^2(v_3+v_4) + 3^737^1b^4(v_3 + v_4) - 
    2^73^5b^3w^2(v_3 + v_4 ) + v_1v_2)x^2 + (v_1v_4 + v_2v_3)x + v_3v_4$, where
    
    $v_1$ and $v_2$ are the roots of $p_1=x^2 + (2^73^25^1a^7z + 2^63^35^1a^6b + 2^93^4a^6z^3 + 2^83^6a^5bz^2 - 2^{13}3^2a^5z^2w^2 + 2^63^523^1a^4b^2z - 2^{13}3^3a^4bzw^2 +2^53^611^1a^3b^3 + 2^73^7a^3b^2z^3 - 
    2^{11}3^4a^3b^2w^2 + 2^63^9a^2b^3z^2 - 2^{11}3^5a^2b^2z^2w^2 + 2^33^841^1ab^4z - 2^{11}3^6ab^3zw^2 + 2^23^917^1b^5 - 2^93^7b^4w^2)x + 2^{18}a^15 + 
    2^{12}3^2353^1a^14z^2 + 2^{12}3^3353^1a^13bz + 2^{15}3^453^1a^13z^4 - 2^{20}3^3a^13zw^2 + 2^{10}3^37^1197^1a^{12}b^2 + 2^{16}3^553^1a^{12}bz^3 - 2^{19}3^4a^{12}bw^2 + 2^{16}3^8a^{12}z^6 
    - 2^{19}3^317^1a^{12}z^3w^2 +2^{20}3^3a^{12}w^4 + 2^{12}3^57^117^119^1a^{11}b^2z^2 + 2^{16}3^{10}a^{11}bz^5 - 2^{18}3^517^1a^{11}bz^2w^2 - 2^{21}3^6a^{11}z^5w^2 + 
    2^{12}3^623^143^1a^{10}b^3z + 2^{13}3^817^119^1a^{10}b^2z^4 - 2^{17}3^65^2a^{10}b^2zw^2 - 2^{20}3^75^1a^{10}bz^4w^2 + 2^{24}3^4a^{10}z^4w^4 + 2^{10}3^641^153^1a^9b^4 + 
    2^{14}3^911^113^1a^9b^3z^3 - 2^{16}3^641^1a^9b^3w^2 + 2^{15}3^{11}a^9b^2z^6 - 2^{17}3^7137^1a^9b^2z^3w^2 + 2^{20}3^6a^9b^2w^4 + 2^{25}3^5a^9bz^3w^4 + 
    2^93^931^1239^1a^8b^4z^2 + 2^{15}3^{13}a^8b^3z^5 - 2^{16}3^{10}19^1a^8b^3z^2w^2 - 2^{20}3^9a^8b^2z^5w^2 + 2^{23}3^7a^8b^2z^2w^4 + 2^93^{10}2273^1a^7b^5z + 
    2^{11}3^{11}593^1a^7b^4z^4 - 2^{15}3^{10}41^1a^7b^4zw^2 - 2^{19}3^{10}5^1a^7b^3z^4w^2 + 2^{23}3^7a^7b^3zw^4 + 2^{23}3^7a^7b^2z^4w^4 + 2^73^95^17^111^129^1a^6b^6 + 
    2^{12}3^{12}233^1a^6b^5z^3 - 2^{14}3^{10}41^1a^6b^5w^2 + 2^{12}3^{14}a^6b^4z^6 - 2^{15}3^{10}257^1a^6b^4z^3w^2 + 2^{17}3^817^1a^6b^4w^4 + 2^{24}3^8a^6b^3z^3w^4 + 
    2^83^{11}15797^1a^5b^6z^2 + 2^{12}3^{16}a^5b^5z^5 - 2^{14}3^{12}97^1a^5b^5z^2w^2 - 2^{17}3^{12}a^5b^4z^5w^2 + 2^{22}3^{10}a^5b^4z^2w^4 + 
    2^83^{12}5^129^2a^4b^7z + 2^93^{13}863^1a^4b^6z^4 - 2^{13}3^{12}179^1a^4b^6zw^2 - 2^{16}3^{13}5^1a^4b^5z^4w^2 + 2^{22}3^{10}a^4b^5zw^4 + 
    2^{20}3^{10}a^4b^4z^4w^4 + 2^63^{12}5^129^141^1a^3b^8 + 2^{10}3^{14}17^119^1a^3b^7z^3 - 2^{12}3^{13}7^2a^3b^7w^2 - 2^{13}3^{12}13^129^1a^3b^6z^3w^2 + 2^{16}3^{11}11^1a^3b^6w^4 + 
    2^{21}3^{11}a^3b^5z^3w^4 + 2^43^{14}5^21097^1a^2b^8z^2 - 2^{12}3^{14}137^1a^2b^7z^2w^2 + 2^{19}3^{13}a^2b^6z^2w^4 + 2^43^{15}5^123^159^1ab^9z - 
    2^{11}3^{15}79^1ab^8zw^2 + 2^{19}3^{13}ab^7zw^4 + 2^23^{15}7^113^197^1b^{10} - 2^{10}3^{15}59^1b^9w^2 + 2^{12}3^{14}19^1b^8w^4$
   
   and $v_3$ and $v_4$ are the roots of $p_2=x^2 + (2^75^1a^9 - 2^93^2a^8z^2 - 2^93^3a^7bz  - 2^{12}3^17^1a^7zw^2 + 2^53^4a^6b^2 - 2^{11}3^27^1a^6bw^2 + 2^{17}a^6w^4 -2^83^5a^5b^2z^2 - 2^83^6a^4b^3z - 
    2^{11}3^47^1a^4b^2zw^2 - 2^33^8a^3b^4 - 2^{10}3^57^1a^3b^3w^2 + 2^{16}3^3a^3b^2w^4 - 2^53^8a^2b^4z^2 - 2^53^9ab^5z - 2^83^77^1ab^4zw^2 - 2^13^97^1b^6 - 
   2^73^87^1b^5w^2 + 2^{13}3^6b^4w^4)x + 2^{12}5^2a^18 +2^{15}3^5a^17z^2 + 2^{15}3^6a^16bz + 2^{16}3^511^1a^16z^4 - 2^{18}3^15^17^1a^16zw^2 + 2^{11}3^47^119^1a^15b^2 + 
    2^{17}3^611^1a^15bz^3 - 2^{17}3^25^17^1a^15bw^2 +2^{20}3^6a^15z^6 - 2^{20}3^329^1a^15z^3w^2 + 2^{23}5^1a^15w^4 + 2^{13}3^87^2a^14b^2z^2 + 2^{20}3^8a^14bz^5 - 
    2^{19}3^529^1a^14bz^2w^2 - 2^{22}3^5a^14z^5w^2 + 2^{24}3^217^1a^14z^2w^4 + 2^{13}3^859^1a^13b^3z + 2^{16}3^871^1a^13b^2z^4 - 2^{16}3^423^153^1a^13b^2zw^2 - 
    2^{21}3^65^1a^13bz^4w^2 + 2^{24}3^317^1a^13bzw^4 - 2^{28}3^17^1a^13zw^6 + 2^83^72789^1a^{12}b^4 + 2^{17}3^931^1a^{12}b^3z^3 - 2^{15}3^5523^1a^{12}b^3w^2 + 
    2^{18}3^{10}a^{12}b^2z^6 - 2^{20}3^659^1a^{12}b^2z^3w^2 + 2^{21}3^3127^1a^{12}b^2w^4 - 2^{27}3^27^1a^{12}bw^6 + 2^{32}a^{12}w^8 + 2^{12}3^{11}173^1a^{11}b^4z^2 + 
    2^{18}3^{12}a^{11}b^3z^5 - 2^{19}3^913^1a^{11}b^3z^2w^2 - 2^{20}3^9a^{11}b^2z^5w^2 +2^{24}3^517^1a^{11}b^2z^2w^4 + 2^{12}3^{11}151^1a^{10}b^5z + 
    2^{13}3^{12}131^1a^{10}b^4z^4 - 2^{15}3^743^161^1a^{10}b^4zw^2 - 2^{19}3^{10}5^1a^{10}b^3z^4w^2 + 2^{24}3^617^1a^{10}b^3zw^4 - 2^{28}3^47^1a^{10}b^2zw^6 +  2^83^95^27^147^1a^9b^6 + 2^{14}3^{14}17^1a^9b^5z^3 - 2^{14}3^823^141^1a^9b^5w^2 + 2^{16}3^{13}a^9b^4z^6 - 2^{17}3^{10}89^1a^9b^4z^3w^2 + 2^{20}3^6229^1a^9b^4w^4 - 
    2^{27}3^57^1a^9b^3w^6 + 2^{32}3^3a^9b^2w^8 + 2^{10}3^{14}13^129^1a^8b^6z^2 + 2^{16}3^{15}a^8b^5z^5 - 2^{16}3^{12}7^2a^8b^5z^2w^2 - 
    2^{18}3^{12} a^8b^4z^5w^2 + 2^{21}3^917^1a^8b^4z^2w^4 + 2^{10}3^{15}97^1a^7b^7z + 2^{12}3^{14}191^1a^7b^6z^4 - 2^{13}3^{10}41^1107^1a^7b^6zw^2 -    2^{17}3^{13}5^1a^7b^5z^4w^2 + 2^{21}3^{10}17^1a^7b^5zw^4 - 2^{25}3^87^1a^7b^4zw^6 + 2^43^{13}18701^1a^6b^8 + 2^{13}3^{15}71^1a^6b^7z^3 -  2^{12}3^{11}5^17^141^1a^6b^7w^2 +2^{14}3^{15}a^6b^6z^6 - 2^{16}3^{12}7^117^1a^6b^6z^3w^2 + 2^{18}3^9331^1a^6b^6w^4 - 2^{24}3^97^1a^6b^5w^6 + 
    2^{29}3^7a^6b^4w^8 + 2^73^{17}661^1a^5b^8z^2 + 2^{14}3^{17}a^5b^7z^5 - 2^{15}3^{14}59^1a^5b^7z^2w^2 - 2^{16}3^{14}a^5b^6z^5w^2 + 
   2^{20}3^{11}17^1a^5b^6z^2w^4 +2^73^{17}479^1a^4b^9z + 2^83^{17}251^1a^4b^8z^4 - 2^{10}3^{13}17^1383^1a^4b^8zw^2 - 2^{15}3^{15}5^1a^4b^7z^4w^2 + 
    2^{20}3^{12}17^1a^4b^7zw^4 - 2^{24}3^{10}7^1a^4b^6zw^6 + 2^33^{16}7^113^179^1a^3b^{10} + 2^93^{18}7^113^1a^3b^9z^3 - 2^93^{14}1999^1a^3b^9w^2 - 
   2^{12}3^{15}149^1a^3b^8z^3w^2 +2^{15}3^{12}433^1a^3b^8w^4 - 2^{23}3^{11}7^1a^3b^7w^6 + 2^{28}3^9a^3b^6w^8 + 2^53^{20}5^141^1a^2b^{10}z^2 - 
    2^{11}3^{18}23^1a^2b^9z^2w^2 + 2^{16}3^{14}17^1a^2b^8z^2w^4 + 2^53^{20}11^113^1ab^{11}z - 2^83^{16}7^1257^1ab^{10}zw^2 + 2^{16}3^{15}17^1ab^9zw^4 - 
    2^{20}3^{13}7^1ab^8zw^6 + 3^{18}13^273^1b^{12} - 2^73^{17}17^131^1b^{11}w^2 +2^{13}3^{15}107^1b^{10}w^4 - 2^{19}3^{14}7^1b^9w^6 + 2^{24}3^{12}b^8w^8$.

\begin{center} \textbf{Appendix B} \end{center}
\label{sec: appendixB}

The map $k_g \rightarrow k_1$ given by $\alpha \mapsto \beta = \frac{-1}{3^2b}(\alpha^3-4a\alpha^2+(8a^2-72bz)\alpha +2^4 3 d)$ is an isomorphism, where $\alpha$ is a root of $g$. In this appendix, we show that $\beta$ satisfies the defining polynomial $h_1$ of $k_1$.

\begin{equation*}
\begin{split}
3^8b^4h_1(\beta) &= \alpha^{12} - 2^4a\alpha^{11} + (2^55a^2 - 2^53^2bz)\alpha^{10} + (-2^8a^3 + 2^73^3abz + 2^63^4b^2)\alpha^9 \\
& + (-2^817a^4 - 2^{10}3^3a^2bz 
- 2^83^5ab^2 + 2^73^5b^2z^2)\alpha^8 \\
&+ (2^{13}7a^5 -2^{12}3^2a^3bz + 2^{11}3^5a^2b^2 - 2^{10}3^5ab^2z^2 - 
    2^93^7b^3z)\alpha^7  \\
    &+ (-2^{14}5a^6 + 2^{15}3^3a^4bz + 2^{11}3^5a^3b^2 + 2^{11}3^6a^2b^2z^2 + 2^{12}3^7ab^3z + 
    2^{12}3^7b^4 - 2^{11}3^6b^3z^3)\alpha^6 \\
    &+ (-2^{17}5a^7 - 2^{18}3^3a^5bz - 2^{14}3^411a^4b^2 + 2^{15}3^5a^3b^2z^2 
    - 2^{13}3^8a^2b^3z - 2^{15}3^7ab^4 +\\
    &  2^{13}3^6ab^3z^3    + 2^{12}3^9b^4z^2)\alpha^5 + (2^{16}127a^8 - 2^{18}3^25a^6bz
    + 2^{15}3^65a^5b^2 - 2^{15}3^55a^4b^2z^2 -\\
    &  2^{15}3^611a^3b^3z + 2^{16}3^8a^2b^4 - 2^{15}3^6a^2b^3z^3 - 
    2^{14}3^9ab^4z^2 - 2^{16}3^9b^5z + 2^{12}3^8b^4z^4)\alpha^4 \\
    & + (2^{20}3a^9  + 2^{20}3^4a^7bz + 
    2^{21}3^4a^6b^2 + 2^{18}3^6a^5b^2z^2 + 2^{18}3^67a^4b^3z + 2^{16}3^713a^3b^4\\
    &  - 2^{17}3^7a^3b^3z^3 
     - 2^{17}3^6a^3b^3w^2 + 2^{16}3^9a^2b^4z^2 + 2^{18}3^9ab^5z + 2^{15}3^{11}b^6 - 2^{15}3^{10}b^5z^3 \\
    &- 2^{15}3^9b^5w^2)\alpha^3 + (-2^{21}3^3a^{10} - 2^{21}3^5a^8bz - 2^{19}3^431a^7b^2 + 2^{19}3^7a^6b^2z^2\\
    & - 2^{19}3^617a^5b^3z - 2^{18}3^717a^4b^4 + 2^{19}3^6a^4b^3w^2 + 2^{17}3^817a^3b^4z^2 - 
    2^{20}3^9a^2b^5z \\
    &- 2^{17}3^{11}ab^6 + 2^{17}3^9ab^5w^2 + 2^{18}3^{11}b^6z^2)\alpha^2 + 
     (2^{24}3^3a^{11} - 2^{23}3^5a^9bz +\\
     & 2^{25}3^5a^8b^2 - 2^{24}3^7a^6b^3z + 2^{20}3^87a^5b^4 - 2^{21}3^6a^5b^3w^2 - 
    2^{19}3^{10}7a^3b^5z + \\
    & 2^{20}3^8a^3b^4zw^2 + 2^{19}3^{11}a^2b^6 - 2^{19}3^9a^2b^5w^2 - 
    2^{18}3^{13}b^7z + 2^{18}3^{11}b^6zw^2)\alpha \\
    & + 2^{24}3^4a^{12} + 2^{23}3^65a^9b^2 + 2^{22}3^{10}a^6b^4 - 
    2^{23}3^7a^6b^3w^2 -2^{20}3^{10}a^5b^4z^2  \\
    & + 2^{20}3^9a^4b^4zw^2 +2^{18}3^{11}19a^3b^6 + 
    2^{20}3^{11}a^3b^5z^3 - 2^{20}3^{10}5a^3b^5w^2 \\
    &- 2^{18}3^{13}a^2b^6z^2 + 2^{18}3^{12}ab^6zw^2 + 
    2^{18}3^{14}b^8 + 2^{18}3^{14}b^7z^3 - 2^{18}3^{14}b^7w^2
\end{split}
\end{equation*}

    After, reducing using that $\alpha^4=
4\Delta \alpha+12 a \Delta$, this polynomial can be written as $c_3 \alpha^3+c_2\alpha^2+c_1 \alpha+c_0$ where the coefficients are the following:

$c_3= 2^{17}3^6a^4b^3z + 2^{17}3^6a^3b^4 + 2^{17}3^6a^3b^3z^3 - 2^{17}3^6a^3b^3w^2 + 2^{15}3^9ab^5z + 2^{15}3^9b^6 + 2^{15}3^9b^5z^3 - 2^{15}3^9b^5w^2$\\
 
$c_2=-2^{19}3^6a^5b^3z - 2^{19}3^6a^4b^4 - 2^{19}3^6a^4b^3z^3 + 2^{19}3^6a^4b^3w^2 - 2^{17}3^9a^2b^5z - 2^{17}3^9ab^6 - 2^{17}3^9ab^5z^3 + 2^{17}3^9ab^5w^2$\\
 
 $c_1=2^{21}3^6a^6b^3z + 2^{21}3^6a^5b^4 + 2^{21}3^6a^5b^3z^3 - 
    2^{21}3^6a^5b^3w^2 - 2^{20}3^8a^4b^4z^2 + 2^{19}3^8a^3b^5z - 2^{20}3^8a^3b^4z^4 + 
    2^{20}3^8a^3b^4zw^2 + 2^{19}3^9a^2b^6 + 2^{19}3^9a^2b^5z^3 - 2^{19}3^9a^2b^5w^2 - 2^{18}3^{11}ab^6z^2 - 2^{18}3^{11}b^7z - 2^{18}3^{11}b^6z^4 + 2^{18}3^{11}b^6zw^2$\\
    
 $c_0=2^{23}3^7a^7b^3z + 2^{23}3^7a^6b^4 + 2^{23}3^7a^6b^3z^3 - 2^{23}3^7a^6b^3w^2 -  2^{20}3^9a^5b^4z^2 + 2^{21}3^97a^4b^5z - 2^{20}3^9a^4b^4z^4 + 2^{20}3^9a^4b^4zw^2 +   2^{20}3^{10}5a^3b^6 +  2^{20}3^{10}5a^3b^5z^3 - 2^{20}3^{10}5a^3b^5w^2 - 2^{18}3^{12}a^2b^6z^2 + 2^{21}3^{12}ab^7z - 2^{18}3^{12}ab^6z^4 + 2^{18}3^{12}ab^6zw^2 + 2^{18}3^{14}b^8 + 2^{18}3^{14}b^7z^3 - 2^{18}3^{14}b^7w^2$

\newpage 

\bibliography{bib}
\bibliographystyle{plain}

\end{document}